\documentclass[11pt]{amsart}
\usepackage{graphicx} 
\usepackage{amsmath}  
\usepackage{amssymb}  
\usepackage{amsthm}   
\usepackage{array}    

\newtheorem{theorem}{Theorem}

\title{On the Expected Duration of a Generalized Bingo Game}
\author{Vu Phan and Ilie Ugarcovici}
\address{DePaul University, Chicago, IL 60614, USA}
\email{vphan8@depaul.edu,iugarcov@depaul.edu}
\date{November 25, 2025}

\begin{document}

\begin{abstract}
We investigate the expected number of calls required to achieve Bingo in a generalized $(n,m)$-Bingo game, where each $n \times n$ card is filled by sampling $n$ numbers from $m$ possible values per column. Using the inclusion-exclusion principle, we derive exact formulas for the probability distribution and the expected game length. Our main theoretical result proves that the expected number of calls is a linear function of $m$. 
\end{abstract}
\maketitle

\section{Introduction}

Bingo is a fun and entertaining family game. There are multiple players, each with a bingo card as illustrated in Figure~\ref{fig:card} and a caller whose role is to pick numbers out of a bag and say them out loud. If the called number exists on the card, the player will mark it, and if the marked numbers complete a line which is either horizontal, vertical, or diagonal, Bingo is achieved. The game can accommodate many simultaneous players. However, when there are many players, the game ends quickly, since there is a higher chance that the called numbers will land on a line on one of the Bingo cards. Interestingly enough, when there are many people in a Bingo game, it is more likely that the winning card has a completed row than a column (see \cite{Benjamin2017BingoParadox}).

\begin{figure}[htb]
    \centering
    \includegraphics[width=0.4\linewidth]{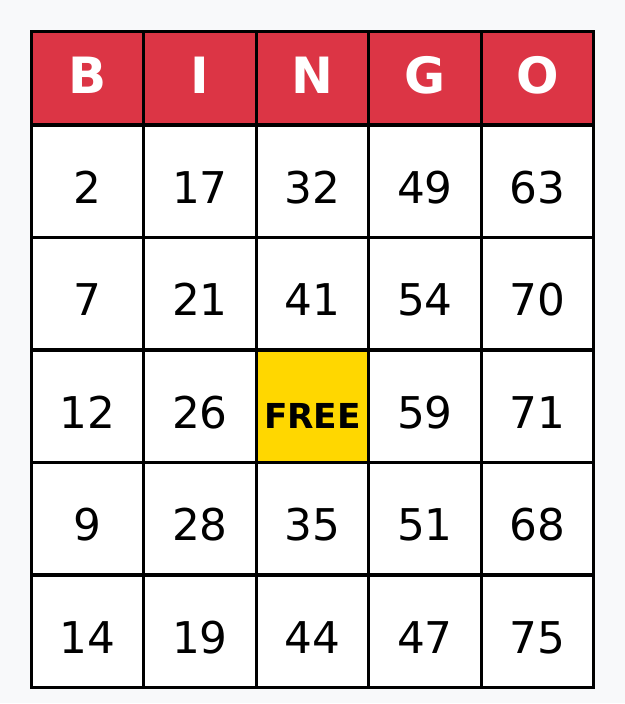}
    \caption{A $5 \times 5$ Bingo Card}
    \label{fig:card}
\end{figure}

In this paper, we investigate how the size $n\times n$ of the Bingo card and the range of call numbers affect the length of the game, in particular the average number of calls required to achieve Bingo. We prove that the expected number of calls to achieve Bingo is linear with respect to the column range.

\section{Mathematical Framework}
The classical Bingo game involves a $5 \times 5$ card with a ``free space" in the center. Each column contains numbers from specific ranges: the first column contains numbers from the interval 1--15, the second column from 16--30, etc. We generalize this to an $(n,m)$-Bingo game, where a card has size $n \times n$ (with $n$ odd) and numbers in each column are selected from $m$ possible values ($m\ge n)$: column $j$ ($1\le j\le n$) contains numbers from the interval $[m\cdot (j-1)+1, m\cdot j]$. The standard Bingo card corresponds to $(n,m) = (5,15)$.

Our primary objective is to determine the expected number of calls required to complete a Bingo card, completing a full row, column, or diagonal. In the article \cite{2002bingo}, D. Agard and M. Shackleford used the inclusion-exclusion principle to compute the cumulative density function (CDF) exactly. To employ this method for an $(n,m)$-Bingo card we need to consider all Bingo winning situations that can occur from $2n + 2$ possible winning patterns: $n$ rows, $n$ columns, and $2$ diagonals.

Let $\mathcal{L}$ denote the set of all $2n+2$ possible winning patterns and, for any subset $\mathbf{X} \subseteq \mathcal{L}$, we define:
\begin{itemize}
\item $|\mathbf{X}|$ = number of Bingo patterns in $\mathbf{X}$,
\item $m(\mathbf{X})$ = number of distinct squares covered by union of patterns in~$\mathbf{X}$.
\end{itemize}

The cumulative probability distribution for a Bingo event, $\mathbf{B}$,  is
\[
P(\mathbf{B} \leq k ) = P\left(\bigcup_{L\in\mathcal{L}}(L\leq k)\right).
\]
By the inclusion-exclusion principle, the above formula can be rewritten as:
\begin{align*}
P(\mathbf{B} \leq k) &= \sum_{\emptyset \neq \mathbf{X} \subseteq \mathcal{L}}(-1)^{|\mathbf{X}|+1}P\left(\bigcap_{L \in \mathbf{X}}(L\leq k)\right).
\end{align*}
Here, $P(\bigcap_{L \in \mathbf{X}}(L\leq k))$ is the probability that all the squares in the union of the lines in $\mathbf{X}$ are among the $k$ called numbers, thereby 
\begin{align*}
P(\mathbf{B} \leq k) &=\sum_{\emptyset \neq \mathbf{X} \subseteq \mathcal{L}}
(-1)^{|\mathbf{X}|+1}\frac{\binom{mn-m(\mathbf{X})}{k-m(\mathbf{X})}}{\binom{mn}{k}} \\
&= \sum_{\emptyset \neq \mathbf{X} \subseteq \mathcal{L}}(-1)^{|\mathbf{X}|+1}\frac{\binom{k}{m(\mathbf{X})}}{\binom{mn}{m(\mathbf{X})}},
\end{align*}
where the last equality follows from:
\begin{align*}
\frac{\binom{mn-m(\mathbf{X})}{k-m(\mathbf{X})}}{\binom{mn}{k}}
&= \frac{(mn-m(\mathbf{X}))! \cdot k!}{(k-m(\mathbf{X}))! \cdot (mn)!} \\
&= \frac{(mn-m(\mathbf{X}))! \cdot k! \cdot m(\mathbf{X})!}{(k-m(\mathbf{X}))! \cdot (mn)! \cdot m(\mathbf{X})!} \\
&= \frac{\frac{k!}{(k-m(\mathbf{X}))!\cdot m(\mathbf{X})!}}
        {\frac{(mn)!}{(mn-m(\mathbf{X}))!\cdot m(\mathbf{X})!}}\\
&= \frac{\binom{k}{m(\mathbf{X})}}{\binom{mn}{m(\mathbf{X})}}.
\end{align*}
Now we compute the probability mass function (PMF): 
\begin{align*}
P(\mathbf{B}=k) &= P(\mathbf{B} \leq k) - P(\mathbf{B} \leq k-1) \\
&= \sum_{\emptyset \neq \mathbf{X} \subseteq \mathcal{L}}(-1)^{|\mathbf{X}|+1}\frac{\binom{k}{m(\mathbf{X})}-\binom{k-1}{m(\mathbf{X})}}{ \binom{mn}{m(\mathbf{X})}}\\
&= \sum_{\emptyset \neq \mathbf{X} \subseteq \mathcal{L}}(-1)^{|\mathbf{X}|+1}\frac{ \binom{k-1}{m(\mathbf{X})-1}}{ \binom{mn}{m(\mathbf{X})}} \\
&= \sum_{\emptyset \neq \mathbf{X} \subseteq \mathcal{L}}(-1)^{|\mathbf{X}|+1}\frac{m(\mathbf{X}) \cdot \binom{k}{m(\mathbf{X})}}{k \cdot \binom{mn}{m(\mathbf{X})}},
\end{align*}
where the second equation is derived using the famous binomial identity: $\binom{k}{n}-\binom{k-1}{n}=\binom{k-1}{n-1}$. Thus the expected value or the expected number of calls to complete a Bingo game is:
\begin{align}\label{eq:expect}
\mathbb{E}[\mathbf{B}] = \sum_{k=1}^{mn} k \cdot P(\mathbf{B}=k) = \sum_{k=1}^{mn} \sum_{\emptyset \neq \mathbf{X} \subseteq \mathcal{L}}(-1)^{|\mathbf{X}|+1} \frac{m(\mathbf{X}) \cdot \binom{k}{m(\mathbf{X})}}{\binom{mn}{m(\mathbf{X})}}.
\end{align}

This formula allowed us to perform some computational experiments to understand the behavior of $\mathbb{E}[\mathbf{B}]$ with respect to $m$ and $n$. 
For a fixed size $n$, we observed that $\mathbb{E}[\mathbf{B}]$ depends linearly on $m$. This is illustrated by Figure \ref{fig:linearity}, where the two straight lines (one for each $n$) visually confirm the linear trend and the distinct slopes. 
For a fixed $m$, similar numerical experiments suggest that $\mathbb{E}[\mathbf{B}]$ depends non-linearly on $n$, as one can also notice from the plots in Figure \ref{fig:linearity}.

\begin{figure}[htb]
    \centering
    \includegraphics[width=0.8\textwidth]{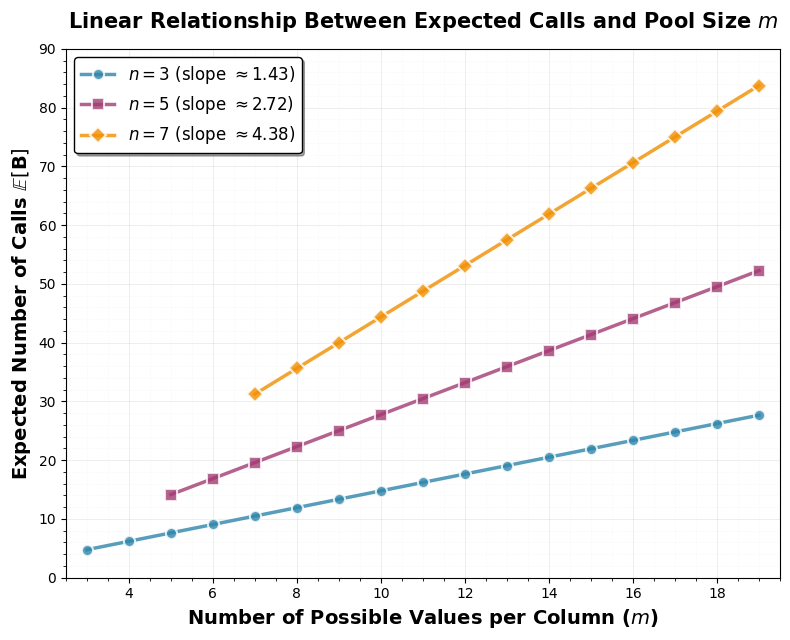}
    \caption{Expected number of calls $\mathbb{E}[\mathbf{B}]$ versus $m$ for $n=5$ and $n=7$.}
    \label{fig:linearity}
\end{figure}

In the next section we prove the linearity property of $\mathbb{E}[\mathbf{B}]$ with respect to $m$.

\pagebreak[4]
\section{Bingo expectation for a single player/card}
We have the following result:



\begin{theorem}
The expected number of calls to complete a $(n,m)$-Bingo card is linear in $m$. More precisely,
\[
\mathbb{E}[\mathbf{B}] = (mn + 1)(1 - S_n)
\]
where
\[
S_n = \sum_{\emptyset \neq \mathbf{X} \subseteq \mathcal{L}}(-1)^{|\mathbf{X}|+1} \cdot \frac{1}{m(\mathbf{X})+1}.
\]
(The alternating sum is over all non-empty subsets of winning patterns.)
\end{theorem}

\begin{proof}
Starting with the expectation formula \eqref{eq:expect}, we change the order of summation:
\begin{align*}
\mathbb{E}[\mathbf{B}] &= \sum_{k=1}^{mn} \sum_{\emptyset \neq \mathbf{X} \subseteq \mathcal{L}}(-1)^{|\mathbf{X}|+1} \frac{m(\mathbf{X}) \cdot \binom{k}{m(\mathbf{X})}}{\binom{mn}{m(\mathbf{X})}} \\
&= \sum_{\emptyset \neq \mathbf{X} \subseteq \mathcal{L}} \frac{m(\mathbf{X}) \cdot (-1)^{|\mathbf{X}|+1}}{\binom{mn}{m(\mathbf{X})}} \sum_{k=m(\mathbf{X})}^{mn}  \binom{k}{m(\mathbf{X})}.
\end{align*}
Applying the upper‐summation identity $\sum_{k=r}^{n} \binom{k}{r} = \binom{n+1}{r+1}$ (see \cite{GKP:ConcreteMath}), also known as ``hockey-stick identity" to the inner sum, we have:
\begin{align*}
\mathbb{E}[\mathbf{B}] 
    &= \sum_{\emptyset \neq \mathbf{X} \subseteq \mathcal{L}} \frac{(-1)^{|\mathbf{X}|+1} \, m(\mathbf{X})}{\binom{mn}{m(\mathbf{X})}} \cdot \binom{mn+1}{m(\mathbf{X})+1} \\
    &= (mn+1) \sum_{\emptyset \neq \mathbf{X} \subseteq \mathcal{L}} (-1)^{|\mathbf{X}|+1} \cdot \frac{m(\mathbf{X})}{m(\mathbf{X})+1} \\
    &= (mn+1) \left(1 - S_n\right),
\end{align*}
where we used $\frac{m(\mathbf{X})}{m(\mathbf{X})+1} = 1 - \frac{1}{m(\mathbf{X})+1}$ and $\sum_{\emptyset \neq \mathbf{X} \subseteq \mathcal{L}}(-1)^{|\mathbf{X}|+1} = 1$.
One can rewrite
\[
\mathbb{E}[\mathbf{B}] = mn(1-S_n) + (1-S_n)
\]
to show that it is a linear function in $m$ with slope $n(1-S_n)$ and intercept $(1-S_n)$.
\end{proof}
\noindent\textit{Remark.} For fixed $m$, the function is not linear in $n$ since $(1-S_n)$ depends on $n$, which also matches our numerical observations.

\subsection*{Probabilistic Interpretation of $1-S_n$}

To understand the relationship between the size of a Bingo card $n\times n$ and the expected number of calls to complete a Bingo game, we analyze the term $(1-S_n)$. According to Table~\ref{tab:Sn}, the term $(1-S_n)$ increases as $n$ increases, but it remains between 0 and 1.

\begin{table}[h!]
\centering
\begin{tabular}{|c|c|c|}
\hline
$n$ & $S_n$  & $1 - S_n$ \\
\hline
3  & 0.61428571 & 0.38571429 \\
5  & 0.45567666 & 0.54432334 \\
7  & 0.37493088 & 0.62506912 \\
9  & 0.32247144 & 0.67752856 \\
11 & 0.28471901 & 0.71528099 \\
\hline\\
\end{tabular}
\caption{Values of $S_n$ and $1 - S_n$ for various $n$.}
\label{tab:Sn}
\end{table}

Using that
\[
\int_0^1 p^{m(\mathbf{X})}\, dp = \frac{1}{m(\mathbf{X})+1},
\]
we have
\begin{align*}
S_n &= \sum_{\emptyset \neq \mathbf{X} \subseteq \mathcal{L}} (-1)^{|\mathbf{X}|+1} \int_0^1 p^{m(\mathbf{X})}\,dp \\
&= \int_0^1 \left( \sum_{\emptyset \neq \mathbf{X} \subseteq \mathcal{L}} (-1)^{|\mathbf{X}|+1} p^{m(\mathbf{X})} \right) dp.
\end{align*}
We can interpret $p$ as the probability that an arbitrary square is independently``marked''. Then, for each subset $\mathbf{X} \subseteq \mathcal{L}$, the probability that all lines in $X$ are fully marked is $p^{m(X)}$. By the inclusion-exclusion principle, the probability that at least one line is fully marked (i.e., at least one Bingo line) is:
\[
P(p) = \sum_{\emptyset \neq \mathbf{X} \subseteq \mathcal{L}} (-1)^{|\mathbf{X}|+1} p^{m(\mathbf{X})}
\]
So,
\begin{equation*}
S_n = \int_0^1 P(p)\,dp \implies 1-S_n = \int_0^1 [1-P(p)]\,dp
\end{equation*}
$P(p)$ is the probability that at least one line is fully marked, so its complement $1-P(p)$ is the probability that no line is fully marked. Let us call this new probability $Q(p)$. We have $1-S_n = \int_0^1 Q(p)\,dp$ and $ 0 < Q(p) < 1 \text{ for all } p \in (0,1)$. Therefore, $0 < 1-S_n < 1$.

To give perspective,
\begin{align*}
    1-S_n = \int_0^1 Q(p)\,dp = \mathbb{E}[Q(p)]
\end{align*}
which is the expected value of $Q(p)$. 
\begin{align*}
    \mathbb{E}[\mathbf{B}] = (mn+1)(1-S_n) = (mn+1)\mathbb{E}[Q(p)]
\end{align*}
This means that the expected number of turns to complete the Bingo game is given by the product between the
\textit{(Number of available values + 1)} and the \textit{(Average probability that no lines are fully marked with marking probability $p$  per square)}.

\section{Bingo expectation for multpliple players/cards}

A natural attempt to generalize to \(N\) players (or cards) is to assume the independence of the cards. However, this assumption does not hold in practice, as there can be overlaps among winning lines of different cards. We therefore proceed to the correct derivation using the inclusion-exclusion principle.

Consider the scenario with \(N\) players, each with a randomly generated Bingo card. The game ends as soon as at least one player achieves Bingo, which occurs when at least one line (winning pattern) on any card is completed. In other words, in the multiplayer setup, the event “Bingo is achieved” is equivalent to the event that at least one Bingo line is completed across all cards. This is essentially the same as the single-player case, where the game ends when at least one line is completed. This observation allows us to extend the single-player probability framework to the multiplayer setting.

Since each Bingo is randomly generated and contains \(2n+2\) lines (for an \(n \times n\) card with rows, columns, and diagonals), it is possible for some Bingo lines to be identical across different cards. Therefore, in the multiplayer case, we must consider only the set of \emph{unique} lines among all cards, unlike the single-card case where all lines are distinct by construction. To be specific, the uniqueness is determined by the set of numbers in each line.

Let \(\mathcal{L}_N\) denote the set of all unique Bingo lines across all \(N\) Bingo cards. For a given realization of the cards (and hence a fixed \(\mathcal{L}_N\)), suppose that after \(k\) calls, a certain combination \(\mathbf{Y} \subseteq \mathcal{L}_N\) of lines are present. We define \(m(\mathbf{Y})\) as the number of unique numbers in the union of winning patterns in \(\mathbf{Y}\) and \(|\mathbf{Y}|\) as the cardinality of \(\mathbf{Y}\).

We have the following:

\begin{theorem}
For \(N\) players with randomly generated Bingo cards, the expected number of calls to complete a \((n,m)\)-Bingo given a realization of the cards is
\[
\mathbb{E}[\mathbf{B}_N \mid \mathcal{L}_N] = (mn+1)(1-S_{n,N}),
\]
where
\[
S_{n,N} = \sum_{\emptyset \neq \mathbf{Y} \subseteq \mathcal{L}_N} (-1)^{|\mathbf{Y}|+1} \frac{1}{m(\mathbf{Y})+1}.
\]
\end{theorem}

\begin{proof}
Conditioning on the card configuration (i.e., on \(\mathcal{L}_N\)), the cumulative distribution function (CDF) for the number of calls required for at least one Bingo line to be completed, denoted \(\mathbf{B}_N\), is:
\begin{align*}
P(\mathbf{B}_N \leq k \mid \mathcal{L}_N) 
&= P\Bigl(\bigcup_{L \in \mathcal{L}_N} (L\leq k) \,\Big|\, \mathcal{L}_N\Bigr) \\
&= \sum_{\emptyset \neq \mathbf{Y} \subseteq \mathcal{L}_N} (-1)^{|\mathbf{Y}|+1} 
   P\Bigl(\bigcap_{L \in \mathbf{Y}} (L \leq k)\,\Big|\, \mathcal{L}_N\Bigr) \\
&= \sum_{\emptyset \neq \mathbf{Y} \subseteq \mathcal{L}_N} (-1)^{|\mathbf{Y}|+1} 
   \frac{\binom{k}{m(\mathbf{Y})}}{\binom{mn}{m(\mathbf{Y})}},
\end{align*}
where the probability term is the probability that all the squares in the union of the lines in \(\mathbf{Y}\) are among the \(k\) called numbers.

As demonstrated in the single-player case, the conditional PMF and expectation are similarly derived as follows:
\begin{align*}
P(\mathbf{B}_N = k \mid \mathcal{L}_N) 
&= \sum_{\emptyset \neq \mathbf{Y} \subseteq \mathcal{L}_N} (-1)^{|\mathbf{Y}|+1} 
   \frac{\binom{k-1}{m(\mathbf{Y})-1}}{\binom{mn}{m(\mathbf{Y})}},
\\
\mathbb{E}[\mathbf{B}_N \mid \mathcal{L}_N] 
&= (mn+1) \sum_{\emptyset \neq \mathbf{Y} \subseteq \mathcal{L}_N} (-1)^{|\mathbf{Y}|+1} 
   \frac{m(\mathbf{Y})}{m(\mathbf{Y})+1}\\
&= (mn+1)\bigl(1-S_{n,N}\bigr)
\end{align*}
where
\[
S_{n,N} = \sum_{\emptyset \neq \mathbf{Y} \subseteq \mathcal{L}_N} (-1)^{|\mathbf{Y}|+1} 
   \frac{1}{m(\mathbf{Y})+1}.\qedhere
\]
\end{proof}

\noindent{Remark.} According to Theorem 2, for any fixed card configuration, \(S_{n,N}\) depends only on the card size \(n\) and the specific set of unique lines \(\mathcal{L}_N\), but is \emph{independent} of \(m\), the number of available numbers in each column. Therefore, for each card configuration, \(\mathbb{E}[\mathbf{B}_N \mid \mathcal{L}_N]\) is a linear function of \(m\) with slope \(n(1-S_{n,N})\) and intercept \((1-S_{n,N})\). In other words, in the multiplayer Bingo game, the expected number of calls until the first Bingo is achieved grows linearly with the column range \(m\), given that the cards are fixed.

\subsection*{Numerical validation} To validate Theorem 2, we compared the analytical formula $\mathbb{E}[\mathbf{B}_N \mid \mathcal{L}_N] = (mn+1)(1-S_{n,N})$ with direct Monte Carlo simulations. For randomly generated card configurations, we simulated complete games by drawing numbers until at least one card achieved Bingo, then compared the average number of calls with the theoretical prediction. Table~\ref{tab:validation} presents the results. The analytical predictions match the simulation results with remarkable precision, 
with relative error 
consistently below 0.05\% for a sample size of 100,000 trials. 

\begin{table}[htb]
\centering
\begin{tabular}{|c|c|c|c|c|c|c|}
\hline
$n$ & $m$ & $N$ & Trials & Simulation & Formula & Rel. Error \\
\hline
3 & 5 & 2 & 100,000 & 6.5176 & 6.5181 & 0.01\% \\
3 & 5 & 3 & 100,000 & 5.9426 & 5.9441 & 0.02\% \\
3 & 7 & 2 & 100,000 & 8.7573 & 8.7546 & 0.03\% \\
5 & 7 & 2 & 100,000 & 18.6488 & 18.6406 & 0.04\% \\
\hline\\
\end{tabular}
\caption{Comparison of simulated and theoretical expected number of calls for various $(n,m,N)$ configurations.}
\label{tab:validation}
\end{table}
\subsection*{Probabilistic Interpretation of $1-S_{n,N}$}
Similarly to the argument from single player game, one can derive the formula
\begin{align*}
    1-S_{n,N} = \int_0^1 Q_N(p)\,dp = \mathbb{E}[Q_N(p)]
\end{align*}
where $Q_N(p)$ is the probability that no line is fully marked across $N$ cards. Thus,
\begin{align*}
    \mathbb{E}[\mathbf{B}_N \mid \mathcal{L}_N] = (mn+1)(1-S_{n,N}) = (mn+1)\mathbb{E}[Q_N(p)],
\end{align*}
with the same interpretation as in the single-player game.

\section{Remarks and Future Directions}

We have established that for a fixed set of $N$ $(n,m)$-Bingo cards, the expected number of calls is linear in $m$ and it is given by the explicit formula $\mathbb{E}[\mathbf{B}_N \mid \mathcal{L}_N] = (mn+1)(1-S_{n,N})$. 

Beyond traditional Bingo, this framework applies to other Bingo variations, where winning patterns are clearly defined and the procedure for calling and marking numbers is specified. For example, in “Four Corners” Bingo, a win is achieved when the four corner numbers are marked. Across multiple cards, the only change in the formula from Theorem 2 is the set of winning patterns, which now consists of the unique lists of four corners corresponding to each card. 


Several questions remain open for future investigation:

\begin{itemize}
  \item \textbf{Asymptotic behavior of \(S_n\)} as \(n \to \infty\): numerical computations suggest that \(S_n\) decreases to $0$, but a more rigorous asymptotic analysis might be useful.
  \item \textbf{Variance analysis:} while we have derived a formula for the expected value of the game length, understanding the variance would provide insights into its unpredictability and possible fluctuations around the mean.
  \item \textbf{Equivalent formulas or fast approximations for expectation:} it would be interesting to find other closed-form relations or faster approximations of \(\mathbb{E}[\mathbf{B}]\). These would speed up numerical explorations and provide other analytic insights.
\end{itemize}

\bibliographystyle{plain}
\bibliography{references.bib}
\end{document}